\documentclass[12pt]{amsart}

\usepackage{amsmath}
\usepackage{amssymb}
\usepackage{amsthm}

\newtheorem{thm}{Theorem}[subsection]
\newtheorem{lem}[thm]{Lemma}
\newtheorem{prop}[thm]{Proposition}
\newtheorem{cor}[thm]{Corollary}

\newtheorem*{conjA}{Conjecture A}
\newtheorem*{questionB}{Question B}
\newtheorem*{propositionC}{Proposition C}
\newtheorem*{theoremD}{Theorem D}

\theoremstyle{definition}
\newtheorem{rem}[thm]{Remark}
\newtheorem{definition}[thm]{Definition}
\newtheorem{example}[thm]{Example}
\newtheorem*{acknowledgements}{Acknowledgements}
\newtheorem*{noteIT}{Note on the computational results given in \cite{I-T}}

\newcommand\Gal{\mathrm{Gal}}
\newcommand\kernel{\mathop{\mathrm{ker}}}
\newcommand\rkzp{\mathop{\mathrm{rank}_{\mathbb{Z}_p}}}
\newcommand\charideal{\mathrm{char}_{\Lambda}}

\begin{document}

\title{On the structure of the Galois group of the 
maximal pro-$p$ extension with restricted ramification over the 
cyclotomic $\mathbb{Z}_p$-extension}
\author{Tsuyoshi Itoh}
\date{\today}

\maketitle

\begin{abstract}
Let $k_\infty$ be the cyclotomic $\mathbb{Z}_p$-extension of an algebraic number field $k$.
We denote by $S$ a finite set of prime numbers which does not contain $p$, and 
$S(k_\infty)$ the set of primes of $k_\infty$ lying above $S$.
In the present paper, we will study the structure of the Galois group $\mathcal{X}_S (k_\infty)$ 
of the maximal pro-$p$ extension unramified outside $S (k_\infty)$ over $k_\infty$.
We mainly consider the question whether $\mathcal{X}_S (k_\infty)$ is 
a non-abelian free pro-$p$ group or not.
In the former part, we treat the case when $k$ is an imaginary quadratic field and $S = \emptyset$ 
(here $p$ is an odd prime 
number which does not split in $k$).
In the latter part, we treat the case when $k$ is a totally real field and $S \neq \emptyset$.
\end{abstract}

\section{Introduction}\label{intro}

\subsection{Maximal unramified pro-$p$ extension over the cyclotomic $\mathbb{Z}_p$-extension}
Let $p$ be an arbitrary prime number, and 
$k_\infty /k$ the cyclotomic $\mathbb{Z}_p$-extension 
of an algebraic number field.
We denote by $\mathcal{L} (k_\infty)/k_\infty$ the 
maximal unramified pro-$p$ extension, and
put $\mathcal{X} (k_\infty) = \Gal (\mathcal{L} (k_\infty)/k_\infty)$.
For the structure of $\mathcal{X} (k_\infty)$, there are many studies 
(\cite{Okano}, \cite{Oza07}, \cite{Sharifi08}, 
\cite{Miz10}, \cite{M-O10}, \cite{Fuj11}, etc.).
Concerning this, the following conjecture is considered 
(see, e.g., \cite[p.298]{Sharifi08}, \cite[p.101]{Fuj11}). 

\begin{conjA}
Let the notation be as above.
Then $\mathcal{X} (k_\infty)$ is not a ``non-abelian free pro-$p$ group''.
\end{conjA}

At this time, there seems to be no counter-examples for Conjecture A 
(see also \cite[p.32, Remark]{Sharifi07}).
In particular, when $p$ splits completely in $k$,  
``Greenberg's generalized conjecture'' \cite[Conjecture (3.5)]{Gre01} for $k$ and $p$
implies Conjecture A for $k$ and $p$ (see \cite[Main Theorem 1.2]{Fuj11}).

Let $L (k_\infty)/k_\infty$ be the maximal unramified abelian pro-$p$ extension.
Then, $X(k_\infty) = \Gal (L (k_\infty)/k_\infty)$ is a fundamental research 
object in Iwasawa theory.
This is also the maximal abelian pro-$p$ quotient of $\mathcal{X} (k_\infty)$, and 
is an important object to study Conjecture A.
Especially, Conjecture A holds when $X(k_\infty)$ has a non-trivial $\mathbb{Z}_p$-torsion 
element (see also \cite[p.104]{Fuj11}).

When $p=2$, M.~Ozaki announced that Conjecture A holds for 
most imaginary quadratic fields.
Inspired by his study, the author conducted the study of the present paper.
(Several ideas in the present paper were learned from Ozaki's talk.)

\subsection{Criteria for imaginary quadratic fields}
Let $k$ be an imaginary quadratic field and $p$ an odd prime number.
Assume that $k$ and $p$ satisfy both of the following conditions:

\begin{itemize}
\item[(C1)] $p$ does not split in $k$,
\item[(C2)] the class number $h(k)$ of $k$ is divisible by $p$.
\end{itemize}

\noindent 
We recall some properties of $X(k_\infty)$ in this case.
Denote by $\lambda (k)$ the Iwasawa $\lambda$-invariant of $k_\infty/k$. 
Since $p$ is odd, 
it is well known that $X(k_\infty)$ is isomorphic to $\mathbb{Z}_p^{\oplus \lambda(k)}$ as a 
$\mathbb{Z}_p$-module (see, e.g., \cite[p.82]{Oza07}, \cite{Was}).
Thus, to verify Conjecture A, we need more information.
Since $k$ satisfies (C2), we see that $\lambda (k) \geq 1$.
(We note that $\mathcal{X} (k_\infty)$ is an abelian free pro-$p$ group 
when $\lambda(k)$ is $0$ or $1$.)

Let $K/k$ be an unramified cyclic extension of degree $p$.
By considering the action of $\Gal(k/\mathbb{Q})$ on 
the Sylow $p$-subgroup of the ideal class group of $k$, we can show that 
$K$ is a Galois extension over $\mathbb{Q}$, and $\Gal (K/\mathbb{Q})$ is 
isomorphic to the dihedral group of order $2p$. 
Let $F$ be an intermediate field of $K/\mathbb{Q}$ of degree $p$.
We denote by $\lambda (F)$ the Iwasawa $\lambda$-invariant of 
the cyclotomic $\mathbb{Z}_p$-extension $F_\infty/F$.
We can obtain the following criterion for Conjecture A 
(the proof will be given in Section \ref{proof_thm_im1}).

\begin{thm}\label{thm_im1}
Let $k$ be an imaginary quadratic field and $p$ an odd prime number.
Assume that $k$ and $p$ satisfy (C1) and (C2). 
Let $K/k$ be an unramified cyclic extension of degree $p$, and 
$F$ an intermediate field of $K/\mathbb{Q}$ of degree $p$.
Assume also that $\lambda (k) \geq 2$.
If 
\[ \lambda (F) \leq \dfrac{(p-1) \lambda (k) - p}{2}, \] 
then $\mathcal{X} (k_\infty)$ is not a free pro-$p$ group.
\end{thm}

When $p=3$, we can also obtain another criterion for Conjecture A.
Assume that $k$ and $3 (=p)$ satisfy (C1) and (C2).
Then, the above $F$ is a cubic (non-Galois) extension over $\mathbb{Q}$.
We note that the unique prime of $k$ lying above $3$ splits completely in $K$.
From this, we see that there are exactly two primes lying above $3$ in $F$, 
and they are totally ramified in $F_\infty$.
For an integer $n \geq 0$, let $F_n$ be the $n$th layer of $F_\infty/F$ (note that $F_0 =F$).
We denote by $\mathfrak{p}_n$, $\mathfrak{p}'_n$ the primes of $F_n$ lying above $3$.
Let $A(F_n)$ be the Sylow $3$-subgroup of the ideal class group of $F_n$.
We put 
\[ D(F_n) = \langle c(\mathfrak{p}_n), c(\mathfrak{p}'_n) \rangle \cap A(F_n), \]
where $c(\mathfrak{p}_n)$ (resp. $c(\mathfrak{p}'_n)$) is the ideal class containing 
$\mathfrak{p}_n$ (resp. $\mathfrak{p}'_n$).
The second criterion is the following:

\begin{thm}\label{thm_im2}
Assume that $p=3$.
Let $k$ be an imaginary quadratic field satisfying (C1) and (C2), 
$K/k$ an unramified cyclic cubic extension, and 
$F$ an intermediate field of $K/\mathbb{Q}$ of degree $3$.
If $D (F_n)$ is not trivial for some $n \geq 0$, then 
$\mathcal{X} (k_\infty)$ is not a free pro-$3$ group.
\end{thm}

The above theorem seems close to Ozaki's result \cite[Proposition 5]{Oza07}.
Note that our proof of the above theorem uses different tools from Ozaki's.
(See also Section \ref{Ozaki}.)
This theorem will be shown in Section \ref{proof_thm_im2}.

\subsection{Several results around Conjecture A}\label{sub_conj1}
In Sections \ref{proof_prop_conj1} and \ref{sec_im_D}, 
we will give results obtained from the study of Conjecture A.

Let $k$ be an imaginary quadratic field and $p$ an odd prime number 
which satisfies (C1) and (C2), and take $K$, $F$ similarly as above.
Separately from the validity of Conjecture A, we are also concerned with 
the structure of $X (F_\infty)$ itself.
We will give a result which relates $X (F_\infty)$ to $\lambda (k)$.
Before stating this, we shall introduce the following additional condition:

\begin{itemize}
\item[(C3)] the Sylow $p$-subgroup $A (k)$ of the ideal class group of $k$ is cyclic.
\end{itemize}

\noindent We will show the following theorem in Section \ref{proof_prop_conj1}.
This theorem can be shown by using Lemmermeyer's results given in \cite{Lem}.

\begin{thm}\label{prop_conj1}
Let $k$ be an imaginary quadratic field and $p$ an odd prime number.
Assume that $k$ and $p$ satisfy (C1), (C2), and (C3).
We denote by $K/k$ the unique unramified cyclic extension of degree $p$, and 
$F$ an intermediate field of $K/\mathbb{Q}$ of degree $p$.
Then, $\lambda (k)=1$ if and only if $X (F_\infty)$ is trivial.
\end{thm}

The above theorem slightly relates to our criteria for Conjecture A 
(Theorems \ref{thm_im1} and \ref{thm_im2}).
See the last paragraph of Section \ref{computation_conjA}.

When $p=3$, Theorem \ref{prop_conj1} can be reformulated as the following.
In this case, $F$ is a cubic field (and it is not totally real).
Let $E(F)$ be the group of units in $F$.
Since the free rank of $E(F)$ is $1$,  
we can take a unit $\varepsilon_F$ which satisfies 
$E(F) = \langle -1, \varepsilon_F \rangle$.
As noted before, there are two primes lying above $3$ in $F$.
We also see that one of such a prime $\mathfrak{p}$ satisfies $F_\mathfrak{p} = \mathbb{Q}_3$
(i.e., the local degree at $\mathfrak{p}$ is $1$).
Then we can obtain the following corollary 
by using Hachimori's result given in \cite{Hachi} 
(for the details, see Section \ref{proof_prop_conj1}).

\begin{cor}\label{conj1'}
Let the assumptions be as in Theorem \ref{prop_conj1}.
Moreover, assume that $p=3$.
Then, $\lambda (k)=1$ if and only if $\varepsilon_F^{2} \not\equiv 1 
\pmod{\mathfrak{p}^2}$.
\end{cor}

This corollary says that a property of a certain unit of $F$ determines whether
$\lambda (k)=1$ or $\lambda (k)>1$ (when $p=3$).
(Note that there is another criterion.
See \cite[Theorem 1]{Kra-Was}. 
See also \cite{H-W}.)

In Section \ref{sec_im_D}, we consider a question whether  
$\mathcal{X} (k_\infty)$ is a Demu{\v s}kin group or not.
For this question, we will give a result similar to Theorem \ref{thm_im1}.

In Section \ref{sec_computation}, we shall give some computational results.

\subsection{$S$-ramified analog of Conjecture A}
We shall return to the general case.
That is, $p$ is an arbitrary prime number, 
and $k_\infty /k$ the cyclotomic $\mathbb{Z}_p$-extension 
of an algebraic number field.
Let $S$ be a finite set of prime numbers, and $S(k)$ (resp. $S (k_\infty)$) the set of primes of 
$k$ (resp. $k_\infty$) lying above the prime numbers in $S$.
Assume that $S$ does not contain $p$.
Let $\mathcal{L}_S (k_\infty)/k_\infty$ be the maximal pro-$p$ 
extension unramified outside $S(k_\infty)$.
(We mention that $\mathcal{L}_S (k_\infty)/k_\infty$ is unramified at archimedean primes.)
Put $\mathcal{X}_S (k_\infty) = \Gal (\mathcal{L}_S (k_\infty)/k_\infty)$.
In this situation, we will consider a question similar to Conjecture A 
(see also \cite[Section 5]{F-I}).

\begin{questionB} 
Let the notation and the assumption be as above (in particular, $p \not\in S$). 
Can $\mathcal{X}_S (k_\infty)$ be a non-abelian free pro-$p$ group?
\end{questionB}

Let $\mathbb{B}_\infty$ be the cyclotomic $\mathbb{Z}_p$-extension $\mathbb{Q}$.
When $p$ is odd, it is known that if $\mathcal{X}_S (\mathbb{B}_\infty)$ is not trivial, then 
it is not a free pro-$p$ group (\cite[Corollary 5.1]{F-I}).
This result follows from the fact that 
the maximal abelian pro-$p$ quotient of $\mathcal{X}_S (\mathbb{B}_\infty)$ 
has a non-trivial $\mathbb{Z}_p$-torsion element when it is not trivial 
(\cite[Theorem 1.1]{F-I}).
Note also that some partial results for imaginary quadratic fields are given in \cite{F-I}.
Moreover, \cite{BLM} also gave results concerning Question B.
(For the structure of $\mathcal{X}_S (k_\infty)$ and related topics, 
see also \cite{Salle08}, \cite{Salle10}, \cite{M-O13}, \cite{I-M}, \cite{Miz18}, etc.)

In the present paper, we shall consider Question B
when $k'$ is a totally real field.
(We use the symbol $k'$ for a totally real field 
because $k$ is used as an imaginary quadratic field in 
Sections \ref{sec_im1} and \ref{sec_computation}.)
First, we will relate the above question to Greenberg's conjecture (see \cite{Gre76}) 
which states that $X (k'_\infty)$ is finite if $k'$ is totally real.
In Section \ref{sec_real_GC}, we shall show the following result.

\begin{thm}\label{relationship_GC}
Let $p$ be an (arbitrary) prime number,
$k'$ a totally real number field, and 
$S$ a finite set of prime numbers which does not contain $p$.
Assume that there is an infinite Galois extension $\mathbb{L}/k'_\infty$ 
unramified outside $S(k'_\infty)$ such that 
$\mathcal{G}= \Gal (\mathbb{L}/k'_\infty)$ is a finitely generated pro-$p$ group 
and satisfies either of the following conditions:
\begin{itemize}
\item[(F)] $\mathcal{G}$ is a free pro-$p$ group with 
$d (\mathcal{G}) \geq 2$, or 
\item[(D)] $\mathcal{G}$ is a Demu{\v s}kin group with 
$d (\mathcal{G}) \geq 3$. 
\end{itemize}
(Here, $d (\mathcal{G})$ denotes the generator rank of $\mathcal{G}$.)
Then, there is a totally real number field 
$k''$ such that $X (k''_\infty)$ is not finite 
(i.e., a counter-example for Greenberg's conjecture exists).
\end{thm}

In particular, if $\mathcal{X}_S (k'_\infty)$ is a finitely generated 
``non-abelian'' free pro-$p$ group 
for some totally real field $k'$, then Greenberg's conjecture fails.
Hence it seems significant to determine whether $\mathcal{X}_S (k'_\infty)$ is a 
free pro-$p$ group or not.
(Note that if $\mathcal{X}_S (k'_\infty)$ is not finitely generated, then 
$X (k'_\infty)$ is not finitely generated as a $\mathbb{Z}_p$-module.
See, e.g., \cite[p.1494]{IMO}.)

In Section \ref{sec_real_quadratic}, 
we shall consider Question B in the case of real quadratic fields (with odd $p$).
Let $k'$ be a real quadratic field, and $k'_\infty/k'$ the cyclotomic $\mathbb{Z}_p$-extension.
We put $X_S (k'_\infty) = \Gal (L_S (k'_\infty)/k'_\infty)$ 
(resp. $X_S (k') = \Gal (L_S (k')/k')$), 
where $L_S (k'_\infty)/k'_\infty$ (resp. $L_S (k')/k'$) is the maximal 
abelian pro-$p$ extension 
unramified outside $S (k'_\infty)$ (resp. $S (k')$).
When $X_S (k'_\infty)$ is trivial or $X_S (k'_\infty)$ contains a non-trivial 
$\mathbb{Z}_p$-torsion element, then $\mathcal{X}_S (k'_\infty)$ is not a non-abelian 
free pro-$p$ group. 
Unfortunately, there is a case that $X_S (k'_\infty)$ is not trivial and 
has no non-trivial $\mathbb{Z}_p$-torsion elements.
(For this topic, see \cite{Itoh18}.)
We will consider the above question for such a case.
In the present paper, we show the following:

\begin{thm}\label{thm_real_quad}
Let $k'$ be a real quadratic field and $p$ an odd prime number.
Assume that $p$ is inert in $k'$, and $p$ does not divide the class number of $k'$.
Let $r$ be a positive integer satisfying $r \geq 3$, and 
$S= \{ q_1, q_2, q_3, \ldots, q_r \}$ a set of $r$ distinct prime numbers.
Suppose that  
\[ \text{$q_i$ is inert in $k'$}, \quad q_i \equiv -1 \pmod{p}, \quad \text{and} \quad 
\quad q_i^2 \not\equiv 1 \pmod{p^2} \]
for $i=1,2,3, \ldots, r$.
Suppose also that both $X_{\{ q_1 \}} (k')$ and $X_{\{ q_2 \}} (k')$ are trivial.
Then $\mathcal{X}_S (k'_\infty)$ is not a free pro-$p$ group.
\end{thm}

Under the assumptions of the above theorem, we see that
$X_S (k'_\infty) \cong \mathbb{Z}_p^{\oplus r-1}$ as a $\mathbb{Z}_p$-module 
(see Lemma \ref{lem_real1}).

\subsection{Notation}
Let $p$ be an arbitrary prime number.
In this subsection, $\mathcal{K}$ is an algebraic extension of $\mathbb{Q}$, 
and $S$ is a finite set of prime numbers not containing $p$.
(We assume that all algebraic extensions of $\mathbb{Q}$ 
are embedded in $\mathbb{C}$ under a fixed embedding.)
Let $S (\mathcal{K})$ be the set of primes of $\mathcal{K}$ lying above $S$.
We will use the following notation:
\begin{itemize}
\item $\mathcal{L}_S (\mathcal{K})/\mathcal{K}$ : the maximal pro-$p$ extension unramified outside 
$S (\mathcal{K})$,
\item $\mathcal{X}_S (\mathcal{K}) =\Gal (\mathcal{L}_S (\mathcal{K})/\mathcal{K})$,
\item $L_S (\mathcal{K})/\mathcal{K}$ : the maximal 
abelian pro-$p$ extension unramified outside 
$S (\mathcal{K})$, and 
\item $X_S (\mathcal{K}) =\Gal (L_S (\mathcal{K})/\mathcal{K})$.
\end{itemize}
For simplicity, we shall write $\mathcal{L} (\mathcal{K})$, 
$\mathcal{X} (\mathcal{K})$, $L (\mathcal{K})$, $X (\mathcal{K})$ 
instead of $\mathcal{L}_\emptyset (\mathcal{K})$, 
$\mathcal{X}_\emptyset (\mathcal{K})$, $L_\emptyset (\mathcal{K})$, 
$X_\emptyset (\mathcal{K})$, respectively.

We denote by $\mathbb{B}_\infty$ the cyclotomic $\mathbb{Z}_p$-extension of $\mathbb{Q}$.
When $\mathcal{K}$ is a finite extension of $\mathbb{Q}$, we will define the following notation:
\begin{itemize}
\item $h(\mathcal{K})$ : the class number of $\mathcal{K}$,
\item $A(\mathcal{K})$ : the Sylow $p$-subgroup of the ideal class group of $\mathcal{K}$ 
($A(\mathcal{K}) \cong X (\mathcal{K})$ by class field theory),
\item $E (\mathcal{K})$ : the group of units in $\mathcal{K}$,
\item $\mathcal{K}_{\mathfrak{p}}$ : the completion of $\mathcal{K}$ at a prime $\mathfrak{p}$ of 
$\mathcal{K}$,  
\item $\mathcal{K}_\infty = \mathcal{K} \mathbb{B}_\infty$ (the cyclotomic $\mathbb{Z}_p$-extension 
of $\mathcal{K}$),
\item $\mathcal{K}_n$ : the $n$th layer of $\mathcal{K}_\infty/\mathcal{K}$ 
($\mathcal{K}_n/\mathcal{K}$ is a cyclic extension of degree $p^n$), 
\item $\lambda (\mathcal{K})$ : the Iwasawa 
$\lambda$-invariant of $\mathcal{K}_\infty/\mathcal{K}$, and 
\item $\mu (\mathcal{K})$ : the Iwasawa 
$\mu$-invariant of $\mathcal{K}_\infty/\mathcal{K}$.
\end{itemize}

For a finite group $G_1$, we denote by $|G_1|$ its order.
For a finite abelian $p$-group $G_2$, we put 
$r_p (G_2) = \dim_{\mathbb{F}_p} G_2 / p G_2$ (the $p$-rank of $G_2$).
For a $\mathbb{Z}_p$-module $M_1$, we put 
$\rkzp M_1 = \dim_{\mathbb{Q}_p} M_1 \otimes_{\mathbb{Z}_p} \mathbb{Q}_p$ 
(the $\mathbb{Z}_p$-rank of $M_1$).
For a cyclic $p$-group (resp. an infinite procyclic pro-$p$ group) $G_3$ and 
a $\mathbb{Z}_p [G_3]$-module (resp. $\mathbb{Z}_p [[G_3]]$-module) $M_2$, 
let $(M_2)^{G_3}$ be the $G_3$-invariant submodule of $M_2$ and 
$(M_2)_{G_3}$ the $G_3$-coinvariant quotient of $M_2$.

\subsection{Free pro-$p$ groups and Demu{\v s}kin groups}
In this subsection, we recall some properties of free pro-$p$ groups and Demu{\v s}kin groups.
(See, e.g., \cite{Serre}, \cite{Koch}, \cite{NSW}.)

Let $p$ be an arbitrary prime number.
For a pro-$p$ group $\mathcal{G}$, we denote by $d( \mathcal{G} ) = 
\dim_{\mathbb{F}_p} H^1 (\mathcal{G}, \mathbb{F}_p)$ its generator rank and 
$\mathcal{G}^{ab}$ its maximal abelian pro-$p$ quotient.

\begin{lem}\label{pro-ell}
Let $\mathcal{G}$ be a finitely generated pro-$p$ group.

\smallskip

\noindent (i) If $\mathcal{G}$ is a free pro-$p$ group satisfying 
$d( \mathcal{G}) \geq 1$, then 
\begin{itemize}
\item $\mathcal{G}^{ab}$ is a free $\mathbb{Z}_p$-module of rank $d( \mathcal{G})$, and
\item for every open subgroup $U$ of $\mathcal{G}$, $U$ is also a free pro-$p$ group and 
\[ d(U) - 1 = (\mathcal{G}:U) (d( \mathcal{G}) - 1). \]
\end{itemize}

\smallskip

\noindent (ii) If $\mathcal{G}$ is a Demu{\v s}kin group satisfying 
$d( \mathcal{G}) \geq 2$, then 
\begin{itemize}
\item $\mathop{\mathrm{rank}_{\mathbb{Z}_p}} \mathcal{G}^{ab}$ 
is $d( \mathcal{G})$ or $d( \mathcal{G})-1$, and
\item for every open subgroup $U$ of $\mathcal{G}$, $U$ is also a Demu{\v s}kin group and 
\[ d(U) - 2 = (\mathcal{G}:U) (d( \mathcal{G}) - 2). \]
\end{itemize}
\end{lem}

\begin{proof}
These assertions are well known.
See, e.g., \cite{Serre}, \cite{Koch}, \cite{NSW}.
\end{proof}

\section{Results for imaginary quadratic fields}\label{sec_im1}

\subsection{Preliminaries ($p \geq 3$)}\label{sub_Settings} 
Let $k$ be an imaginary quadratic field and $p$ an odd prime number 
satisfying (C1) and (C2).
Take an unramified cyclic extension $K/k$ of degree $p$, and let $F$ be an intermediate field of 
$K/\mathbb{Q}$ of degree $p$.
We fix elements $\sigma$, $\tau$ of $\Gal (K/\mathbb{Q})$ such that 
$\sigma$ generates $\Gal (K/k)$ and $\tau$ generates $\Gal (K/F)$.
(Then, $\tau \sigma = \sigma^{-1} \tau$.)

We will give some remarks on $K_\infty /K$ and $X (K_\infty)$.
Since $K \cap \mathbb{B}_\infty = \mathbb{Q}$, we shall identify $\Gal (K_\infty/\mathbb{B}_\infty)$ 
with $\Gal (K/\mathbb{Q})$.
We put $\Gamma = \Gal (K_\infty/K)$, and we also identify $\Gal(F_\infty/F)$, $\Gal(k_\infty/k)$ 
with $\Gamma$.
It is well known that $X (K_\infty)$ is finitely generated torsion $\mathbb{Z}_p [[\Gamma]]$-module.
We see that $\mu (k)=0$ by Ferrero-Washington's theorem \cite{F-W}, 
and we also see that $\mu (K) =0$ because 
$K_\infty /k_\infty$ is a cyclic extension of degree $p$ (see \cite[Theorem 3]{Iwa73mu}).
Hence, $X (K_\infty)$ is finitely generated as a $\mathbb{Z}_p$-module 
(the same assertion also holds for $X (F_\infty)$).
We also note that all primes of $K$ lying above $p$ are totally ramified in $K_\infty$.

We put $X(K_\infty)^{\pm} = (1 \pm \tau) X (K_\infty)$.
Then $X(K_\infty) \cong X(K_\infty)^{+} \oplus X(K_\infty)^{-}$, and 
$X(K_\infty)^{+} \cong X(F_\infty)$.
We shall define a homomorphism (as $\mathbb{Z}_p$-modules) from 
$X(K_\infty)^-$ to $X(K_\infty)^+$ inspired by the idea given in 
Konomi's paper (see \cite[Proposition 1.5 (4)]{Kono}).

\begin{definition}\label{Konomi_hom} \normalfont
Let the notation be as above.
We define a homomorphism $\phi$ by the following:
\[ \phi : X (K_\infty)^- \rightarrow X (K_\infty)^+, \quad x \mapsto 
(\sigma^{\frac{p-1}{2}} - \sigma^{\frac{p+1}{2}}) x. \]
\end{definition}

\begin{lem}\label{Konomi_property}
Let the notation be as above.
Then the image of $\phi$ is certainly contained in $X (K_\infty)^+$, 
and the kernel $\kernel \phi$ of $\phi$ is 
\[ \{ x \in X (K_\infty)^- \, | \, \sigma x = x \}. \]
\end{lem}

\begin{proof}
This lemma follows from direct computations.
\end{proof}

\subsection{First criterion ($p \geq 3$)}\label{proof_thm_im1}

\begin{proof}[Proof of Theorem \ref{thm_im1}] 
Assume that $\mathcal{X} (k_\infty)$ is a free pro-$p$ group and 
\[ \lambda (F) \leq \dfrac{(p-1)\lambda (k)-p}{2}. \]
Put $\mathcal{G} = \mathcal{X} (k_\infty)$ and $U = \mathcal{X} (K_\infty)$.
From the fact that $K/k$ is unramified, 
we see that $\mathcal{L} (K_\infty) = \mathcal{L} (k_\infty)$, and hence 
$U$ is an open subgroup of $\mathcal{G}$ of index $p$.
We shall apply Lemma \ref{pro-ell} (i) for $\mathcal{G}$ and $U$.
Since $X (k_\infty) \cong \mathbb{Z}_p^{\lambda (k)}$, 
the generator rank of $\mathcal{G}$ is $\lambda (k)$. 
We see that $U$ is also a free pro-$p$ group, and
\[ d (U) - 1 = p (\lambda (k) -1 ). \]
Consequently, 
\begin{equation}\label{eq_thm1}
\lambda (K) = \rkzp U^{ab} = d (U) = p \lambda (k) - p +1.
\end{equation}

Let $\phi$ be the homomorphism given in Definition \ref{Konomi_hom}.
We put 
\[ \lambda^+ (K) = \rkzp X(K_\infty)^+ (= \lambda (F)) \quad \text{and} 
\quad \lambda^- (K) = \rkzp X(K_\infty)^-.\]
Consider the following exact sequence 
\[ 0 \to X (K_\infty)^{\langle \sigma \rangle} \to X (K_\infty) 
\overset{\sigma - 1}{\longrightarrow} 
X (K_\infty) \to X (K_\infty)_{\langle \sigma \rangle} \to 0. \]
From this, we see that $\rkzp X (K_\infty)^{\langle \sigma \rangle} 
= \rkzp X (K_\infty)_{\langle \sigma \rangle}$.
We note that $\kernel \phi$ is a $\mathbb{Z}_p$-submodule of 
$X(K_\infty)^{\langle \sigma \rangle}$.
Note also that the intermediate field $L'$ of $L (K_\infty)/K_\infty$ corresponding to 
$(\sigma -1) X(K_\infty)$ is an abelian (unramified) pro-$p$ extension over $k_\infty$ 
(that is, $L'$ is an intermediate field of $L (k_\infty)/k_\infty$).
By combining these facts, we obtain the following inequalities 
\[ \rkzp (\kernel \phi) \leq \rkzp X(K_\infty)^{\langle \sigma \rangle} 
= \rkzp X(K_\infty)_{\langle \sigma \rangle} \leq \lambda (k). \]
Hence, by using Lemma \ref{Konomi_property}, 
we see that 
\[ \lambda^- (K) \leq  \rkzp (\kernel \phi) + \lambda^+ (K) \leq \lambda(k) + \lambda (F). \]
Moreover, we see the following:
\[ \lambda (K) = \lambda^+ (K) + \lambda^- (K) \leq \lambda (k) + 2 \lambda (F). \]
Since $2 \lambda (F) \leq (p-1)\lambda (k)-p$, we see  
\[ \lambda (K) \leq p \lambda (k) - p. \]
This contradicts the equality (\ref{eq_thm1}).
\end{proof}

\subsection{Second criterion ($p = 3$)}\label{proof_thm_im2}
In this subsection, we assume that $p=3$.
Recall that $F$ is a cubic field which is not totally real, and 
there are only two primes lying above $3$ in $F$.

\begin{definition} \label{unit_F} \normalfont 
Let $\mathfrak{p}$ be the prime of $F$ which is lying above $3$ and  
satisfies $F_{\mathfrak{p}} = \mathbb{Q}_3$.
We also take a unit $\varepsilon_F$ such that $E(F) = \langle -1, \varepsilon_F \rangle$. 
We define a positive integer $e (F)$ which satisfies
\[ \varepsilon_F^{2} \equiv 1 \pmod{\mathfrak{p}^{e(F)}} \quad \text{and} \quad 
\varepsilon_F^{2} \not\equiv 1 \pmod{\mathfrak{p}^{e(F)+1}}. \]
Note that $e(F)$ does not depend on the choice of $\varepsilon_F$.
\end{definition}

\begin{proof}[Proof of Theorem \ref{thm_im2}.] 
By using Lemma \ref{pro-ell} (i), we see that 
if $\mathcal{X} (k_\infty)$ is a 
free pro-$3$ group, then $X (K_\infty)$ is a finitely generated free 
$\mathbb{Z}_3$-module 
(recall the first paragraph of the proof of Theorem \ref{thm_im1}).
Hence, $\mathcal{X} (k_\infty)$ is not a free pro-$3$ group when 
$X (K_\infty)$ has a non-trivial $\mathbb{Z}_3$-torsion element.

Recall the facts that $X(K_\infty) \cong X(K_\infty)^{+} \oplus X(K_\infty)^{-}$ and 
$X(K_\infty)^{+} \cong X(F_\infty)$. 
Hence, to see the assertion of this theorem, 
it is sufficient to show that 
$X(F_\infty)$ has a non-trivial $\mathbb{Z}_p$-torsion element 
(cf. \cite{F-I}, \cite{Itoh18}).
The outline of the following proof is similar to \cite[Corollary 3.5 (i)]{F-I}.

Let $\mathfrak{p}$ and $e(F)$ be as in Definition \ref{unit_F}.
We put $\Gamma_n = \Gal (F_n/F)$ for $n \ge 0$.
First of all, we shall show the following inequality:
\begin{equation}\label{ineq_A(F_n)}
|A (F_n)^{\Gamma_n}| \leq |A(F)| \cdot 3^{e(F)-1}.
\end{equation}
We will use the idea given in \cite{Hachi}.
For $n \geq 0$, let $M_{\{ \mathfrak{p} \}} (F_n)$ be the 
maximal abelian pro-$3$ extension of $F_n$ unramified outside $\mathfrak{p}$, 
and put $\mathfrak{X}_{\{ \mathfrak{p} \}} (F_n) = \Gal (M_{\{ \mathfrak{p} \}} (F_n)/F_n)$.
Since there are only two primes in $F$ lying above $3$ and they are totally ramified in $F_\infty/F$,
we can see that 
$\mathfrak{X}_{\{ \mathfrak{p} \}} (F_n)_{\Gamma_n} \cong \mathfrak{X}_{\{ \mathfrak{p} \}} (F)$.
Let $\mathcal{U}^1_{\mathfrak{p}} (F)$ be the group of principal units of $F_{\mathfrak{p}}$, 
that is, 
\[ \mathcal{U}^1_{\mathfrak{p}} (F) = \{ u \, | \, 
\text{$u$ is a unit of $F_{\mathfrak{p}}$ satisfying 
$u \equiv 1 \pmod{\mathfrak{p}}$} \} \]
(we also use the symbol $\mathfrak{p}$ for the maximal ideal of the ring of integers of 
$F_{\mathfrak{p}}$), 
and denote by $\overline{\mathcal{E}}$ the subgroup of 
$\mathcal{U}^1_{\mathfrak{p}} (F)$ topologically generated by $\varepsilon_F^2$.
From the exact sequence 
\[ 0 \to \mathcal{U}^1_{\mathfrak{p}} (F) /\overline{\mathcal{E}}
\to \mathfrak{X}_{\{ \mathfrak{p} \}} (F) \to A (F) \to 0 \]
and the fact that 
$|\mathcal{U}^1_{\mathfrak{p}} (F) /\overline{\mathcal{E}}| = 3^{e(F) -1}$, 
we see that $|\mathfrak{X}_{\{ \mathfrak{p} \}} (F)|= |A(F)| \cdot 3^{e(F) -1}$.
We also note that $A(F_n)$ can be seen as a quotient of $\mathfrak{X}_{\{ \mathfrak{p} \}} (F_n)$, 
then $|\mathfrak{X}_{\{ \mathfrak{p} \}} (F_n)_{\Gamma_n}| \geq |A (F_n)_{\Gamma_n}|$.
From these facts, we see that 
\[ |A (F_n)^{\Gamma_n}| = |A (F_n)_{\Gamma_n}| \leq |\mathfrak{X}_{\{ \mathfrak{p} \}} (F_n)_{\Gamma_n}|
= |\mathfrak{X}_{\{ \mathfrak{p} \}} (F)| = |A(F)| \cdot 3^{e(F) -1}. \]
The inequality (\ref{ineq_A(F_n)}) has been shown.

Assume that an ideal class $c$ of $D(F_{n'})$ is not trivial for some $n'$.
Next, we claim that $c$ becomes trivial in $A (F_m)$ for a sufficiently large $m > n'$.
Since $D(F_n)$ is a subgroup of $A (F_n)^{\Gamma_n}$,
the boundedness of $|A (F_n)^{\Gamma_n}|$ also implies the boundedness of 
$|D(F_n)|$.
Then the claim follows from the argument given in the proof of 
\cite[p.267, Corollary]{Gre76} (with a slight modification).

Finally, by using \cite[p.218, Corollary]{Oza95}, 
we see that the maximal finite $\mathbb{Z}_3 [[\Gamma]]$-submodule 
of $X(F)$ is not trivial.
Then the theorem follows.
\end{proof}

\subsection{Proof of Theorem \ref{prop_conj1} ($p \geq 3$)}\label{proof_prop_conj1}
In this subsection, $p$ denotes an odd prime number.

\begin{propositionC}[see Lemmermeyer \cite{Lem}]
Let $k$ be an imaginary quadratic field and $p$ an odd prime number 
satisfying (C1), (C2), and (C3).
Let $K/k$ be the unique unramified cyclic extension of degree $p$, and 
$F$ an intermediate field of $K/\mathbb{Q}$ of degree $p$.
For a non-negative integer $n$, if $r_p (X (k_n))=1$ then 
$X(F_n)$ is trivial.
(In particular, $X(F)$ is always trivial because $k$ and $p$ satisfy (C3).)
\end{propositionC}

\begin{proof}
This is a special case of \cite[Proposition 8.4]{Lem}. 
(Apply this proposition for $K_n / \mathbb{B}_n$.
Note that the class number of $\mathbb{B}_n$ is prime to $p$, and 
$K_n/k_n$ is an unramified extension.)
\end{proof}

\begin{proof}[Proof of Theorem \ref{prop_conj1}]
First, we shall show the only if part. 
We note that $X(F)$ is trivial by Proposition C.
Assume that $\lambda (k)=1$. 
Since $X (k_\infty) \cong \mathbb{Z}_p$ as a $\mathbb{Z}_p$-module, 
we see that $r_p (X (k_1)) = 1$.
Then, by using Proposition C for $n=1$, 
we see that $X(F_1)$ is trivial.
Since both $X(F)$ and $X(F_1)$ are trivial, we conclude that $X(F_\infty)$ is trivial. 
(See \cite[Theorem 1 (1)]{Fuk94}. 
Note that every prime lying above $p$ is totally ramified in $F_\infty /F$.)

Next, we will show the if part.
Assume that $X (F_\infty)$ is trivial.
Then $X (F_1)$ must be trivial. 
To see the assertion, 
we will show that $r_p (X(k_1)) = 1$ and then apply \cite[Theorem 1 (2)]{Fuk94}.

For a finite extension $\mathcal{K}'/\mathcal{K}$ of algebraic number fields, 
we denote by $N_{\mathcal{K}'/\mathcal{K}}$ the norm mapping 
from $\mathcal{K}'$ to $\mathcal{K}$ 
(elements or ideals).
We use the following result. 
This is a special case of \cite[Theorem 2.1]{Lem}.

\begin{theoremD}[see Lemmermeyer \cite{Lem}]
Let the assumptions be as in Theorem \ref{prop_conj1}.
If $X (F_1)$ is trivial and 
$N_{F_1/\mathbb{B}_1} E(F_1) = E(\mathbb{B}_1)$, then 
$r_p (X(k_1)) =1$.
\end{theoremD}

Hence it is sufficient to show that 
$N_{F_1/\mathbb{B}_1} E(F_1) = E(\mathbb{B}_1)$ 
under the assumption that $X (F_1)$ is trivial.
In the remaining part, we shall show this.

Recall that $p$ does not split in $k$, and 
the unique prime of $k$ lying above $p$ splits completely in $K$.
Then we can take a prime $\mathfrak{p}_F$  of $F$ 
which satisfies $F_{\mathfrak{p}_F} = \mathbb{Q}_p$.
Let $\mathfrak{p}_{F_1}$ be the unique prime of $F_1$ lying above $\mathfrak{p}_F$ 
(note that $\mathfrak{p}_F$ is totally ramified in $F_1/F$).
From our assumption, $h(F_1)$ (the class number of $F_1$) is not divisible by $p$.
Hence there is a positive integer $h$ (which is prime to $p$) such that 
$\mathfrak{p}_{F_1}^h$ is a principal ideal generated by an algebraic integer 
$\eta_{F_1}$ of $F_1$.

Fix a generator $\gamma$ of $\Gal (F_1/F)$ ($\cong \Gal (\mathbb{B}_1/\mathbb{Q})$).
Then $u:=\eta_{F_1}^{\gamma -1}$ is a unit of $F_1$, and 
hence $u' := N_{F_1/\mathbb{B}_1} u$ is a unit of $\mathbb{B}_1$.
Note that $N_{F_1/\mathbb{B}_1} u^{\gamma} =  (u')^{\gamma}$.

We claim that $u' \, E(\mathbb{B}_1)^p$ generates 
$E(\mathbb{B}_1)/E(\mathbb{B}_1)^p$ as 
a $\mathbb{F}_p [ \Gal (\mathbb{B}_1 /\mathbb{Q})]$-module.
Recall that the unique prime $\mathfrak{p}_{\mathbb{B}_1}$ of $\mathbb{B}_1$ 
lying above $p$ is a principal ideal generated by 
$\eta_{\mathbb{B}_1} := N_{\mathbb{Q} (\zeta_{p^2})/\mathbb{B}_1} (1-\zeta_{p^2})$, 
where $\zeta_{p^2}= e^{\frac{2 \pi \sqrt{-1}}{p^2}}$. 
Since $N_{F_1/\mathbb{B}_1} \mathfrak{p}_{F_1} = \mathfrak{p}_{\mathbb{B}_1}$, 
we can write $N_{F_1/\mathbb{B}_1} \eta_{F_1} = \eta_{\mathbb{B}_1}^h \, \varepsilon$ with a 
unit $\varepsilon$ of $\mathbb{B}_1$. 
It is known that $v := \eta_{\mathbb{B}_1}^{\gamma -1}$ generates $E(\mathbb{B}_1)/E(\mathbb{B}_1)^p$ 
as a $\mathbb{F}_p [ \Gal (\mathbb{B}_1 /\mathbb{Q})]$-module 
(see \cite{Sinnott}, \cite{Was}, and \cite[pp.204--205]{Kurihara}).
Since $u' = v^h \, \varepsilon^{\gamma -1}$ 
(and $h$ is prime to $p$), we see that 
$u' \, E(\mathbb{B}_1)^p$ generates 
$E(\mathbb{B}_1)/E(\mathbb{B}_1)^p$ modulo $(E(\mathbb{B}_1)/E(\mathbb{B}_1)^p)^{\gamma -1}$.
Then, by using Nakayama's lemma, the claim can be shown.

Since $N_{F_1/\mathbb{B}_1} E(F_1)$ contains $E(\mathbb{B}_1)^p$, 
the above facts yield that $N_{F_1/\mathbb{B}_1} E(F_1) = E(\mathbb{B}_1)$. 
Then, by using Theorem D, we see that $r_p (X(k_1))= 1 = r_p (X(k))$. 
From \cite[Theorem 1 (2)]{Fuk94} (see also its proof),
we obtain the fact that $X (k_\infty)$ is a (non-trivial) cyclic $\mathbb{Z}_p$-module.
Since $X (k_\infty)$ is not finite in this case, we conclude that 
$\lambda (k)=1$.
\end{proof}

\begin{proof}[Proof of Corollary \ref{conj1'}]
Throughout this proof, we assume that $p=3$.
We use the notation given in Section \ref{proof_thm_im2}.
(Especially, $\mathfrak{p}$ is the unique prime of $F$ lying above $3$ which  
satisfies $F_\mathfrak{p} = \mathbb{Q}_3$.)
In this situation, 
we can apply \cite[Theorem 8.1]{Hachi} for $F$ and $p=3$.
We note that $X (F)$ is trivial by Proposition C.
Thus we see that $X (F_\infty)$ is trivial if and only if 
$\mathcal{U}^1_{\mathfrak{p}} (F) /\overline{\mathcal{E}}$ is trivial.
We also see that 
$\mathcal{U}^1_{\mathfrak{p}} (F) /\overline{\mathcal{E}}$ is trivial 
if and only if $\varepsilon_F^2 \not\equiv 1 \pmod{\mathfrak{p}}$ 
(cf. \cite[Example 8.3]{Hachi}).
The assertion follows.
\end{proof}

\subsection{When is $\mathcal{X} (k_\infty)$ a Demu{\v s}kin group? 
($p \geq 3$)}\label{sec_im_D}
Let $k$ be an imaginary quadratic field and $p$ an odd prime number.
In this subsection, 
we consider the question whether 
$\mathcal{X} (k_\infty)$ is a Demu{\v s}kin group.

We assume that $\lambda(k)=2$ in this paragraph. 
Then $\mathcal{X} (k_\infty)^{ab} \cong \mathbb{Z}_p^{\oplus 2}$.
It can be seen that if a Demu{\v s}kin group $\mathcal{G}$ satisfies 
$d (\mathcal{G}) =2$ and $\mathcal{G}^{ab} \cong \mathbb{Z}_p^{\oplus 2}$, then 
$\mathcal{G}$ must be isomorphic to $\mathbb{Z}_p^{\oplus 2}$. 
(This follows from a well known property of Demu{\v s}kin groups. 
See e.g., \cite[(3.9.1)Theorem]{NSW}.)
Hence, this implies that $\mathcal{X} (k_\infty)$ is a Demu{\v s}kin group
if and only if $\mathcal{X} (k_\infty) = \mathcal{X} (k_\infty)^{ab}$.
Okano \cite{Okano} gave a criterion whether $\mathcal{X} (k_\infty)$ is abelian or not.
In particular, by using \cite[Theorem 1.1]{Okano}, 
if $p$ does not split in $k$ and $\lambda (k) =2$, then
$\mathcal{X} (k_\infty)$ is not a Demu{\v s}kin group.
(On the other hand, when $p$ splits in $k$, 
there is a pair $k$, $p$ such that 
$\mathcal{X} (k_\infty) \cong \mathbb{Z}_p^{\oplus 2}$,
which is an abelian Demu{\v s}kin group.)

We will consider the case when $\lambda (k) \geq 3$.

\begin{rem}\label{odd_lambda}
Since $p$ is odd, the generator rank of a Demu{\v s}kin group must be even 
(see, e.g., \cite[(3.9.16)Proposition]{NSW}).
Hence, if $\lambda (k)$ is odd, then $\mathcal{X} (k_\infty)$ cannot be a Demu{\v s}kin group.
\end{rem}

When $p$ does not split in $k$, we can obtain the following criterion.

\begin{thm}\label{im_non_D}
Let $k$, $K$, and $F$ be as in Theorem \ref{thm_im1}.
Assume that $\lambda (k) \geq 4$ and $\lambda (k)$ is even.
If 
\[ \lambda (F) \leq \dfrac{(p-1) \lambda (k) - 2p}{2}, \] 
then $\mathcal{X} (k_\infty)$ is not a Demu{\v s}kin group.
(It is also not a free pro-$p$ group by Theorem \ref{thm_im1}.)
\end{thm}

\begin{proof}
Suppose that $\mathcal{X} (k_\infty)$ is a Demu{\v s}kin group.
We note that $U = \mathcal{X} (K_\infty)$ is an open subgroup of $\mathcal{X} (k_\infty)$.
Then by Lemma \ref{pro-ell} (ii), we see that $U$ is also a Demu{\v s}kin group and  
\[ d (U) = p \lambda (k) - 2p + 2 \]
(recall that $d (\mathcal{X} (k_\infty)) = \lambda (k)$). 
We note that $d (U) \geq 2$ because $\lambda (k) \geq 4$.
By applying the first assertion of 
Lemma \ref{pro-ell} (ii) for $U$, we see that 
$\lambda (K) = \rkzp U^{ab} \geq d (U) - 1$.
Hence, 
\[ \lambda (K) \geq p \lambda (k) - 2p + 1. \]
On the other hand, if $2 \lambda (F) \leq (p-1) \lambda (k) - 2p$, 
then we obtain the following inequality 
\[ \lambda (K) \leq \lambda (k) + 2 \lambda (F) \leq p \lambda (k) - 2p \]
by using the same argument given in the proof of Theorem \ref{thm_im1}.
It is a contradiction.
\end{proof}

\section{Computations for imaginary quadratic fields}\label{sec_computation}

Let $k$ be an imaginary quadratic field, and put $p = 3$.
Assume that $k$ and $p$ satisfy (C1) and (C2) in Section \ref{intro}.
We will give computational results for Conjecture A.

\subsection{Remarks on computations}
The computations were done by using Magma \cite{MagmaHB} and PARI/GP \cite{PARI}.
For several computations, the author used both Magma and PARI/GP independently and 
checked that these results are the same.
However, for example, the author used PARI/GP only for 
the computation of the Iwasawa $\lambda$-invariant of imaginary quadratic fields 
(because it was due to the function \texttt{Iwapoly} 
made by Y.~Mizusawa \cite{Miz_Iwapoly}).
For the Magma computations, the author referenced the source codes 
written by the co-author of \cite{I-T}.
The author used PARI/GP version 2.9.4 at the time 
that the first manuscript of the present paper was written, however, 
he re-computed the examples by using PARI/GP version 2.11.0 
at the time of revising.

Note that the ideal class groups were computed without assuming GRH 
(except for the computations given in Remark \ref{rem_GRH}).
That is, the author used the command \texttt{ClassGroup} with \texttt{Proof:="Full"} 
for Magma computations, 
and used \texttt{bnfcertify} for PARI/GP computations.

\subsection{Computations for Conjecture A ($p=3$)}\label{computation_conjA}
We assume that $p=3$.
Let $K/k$ be an unramified cubic cyclic extension, and $F$  
an intermediate field of $K/\mathbb{Q}$ of degree $3$.
The author used \texttt{AbelianpExtension} (Magma), 
\texttt{rnfkummer} (PARI/GP) to construct $K$.

Before stating examples, 
we shall give some remarks on the structure of $X (F_\infty)$.
We can show that $\mu (F)=0$ (see Section \ref{sub_Settings}). 
However, it seems a mysterious question whether $\lambda (F) =0$ or not. 
Since $F$ is not totally real, this question lies outside of Greenberg's conjecture. 
On the other hand, $p$ does not split completely in $F$, and hence the result given in 
\cite[p.266]{Gre76} is not applicable 
(i.e., it cannot be said that $\lambda (F)>0$ always holds).
However, we can determine that $\lambda (F)=0$ for some cases by using 
an analog of known criteria (cf., e.g., \cite{Gre76}, \cite{F-K}).

\begin{prop}\label{lambda=0}
Assume that $p=3$.
Let $k$ and $F$ be as in Theorem \ref{thm_im2}, 
$D(F_n)$ as in the paragraph before Theorem \ref{thm_im2}, 
and $e(F)$ as in Definition \ref{unit_F}.
If $|D(F_n)|= |A(F)| \cdot 3^{e(F)-1}$ for some $n \geq 0$, then $\lambda(F)=0$.
\end{prop}

\begin{proof}
Under the assumptions, we see that 
$|A(F_m)^{\Gamma_m}|=|D(F_m)| =|A(F)| \cdot 3^{e(F)-1}$ for all $m \ge n$ by using 
the inequality (\ref{ineq_A(F_n)}) given in the proof of Theorem \ref{thm_im2}. 
and the fact that the mapping $D (F_m) \to D(F_n)$ 
(induced from the norm) is surjective.
Hence by using the argument given in the proof of \cite[Theorem 2]{Gre76}, 
we can conclude that $\lambda (F)=0$.
\end{proof}

\begin{example}\label{ex1}
Put $k=\mathbb{Q} (\sqrt{-211})$ and $p=3$.
Then $k$ and $3$ satisfy (C1) and (C2).
We see that $|A(k)|=3$ and $\lambda (k)=2$.
Then $K$ is the unique unramified cyclic cubic extension over $k$.
In this case, 
we found that $|A(F_1)|=|D(F_1)|=3$ 
(note that the degree of $F_1/\mathbb{Q}$ is $9$).
Then, by Theorem \ref{thm_im2}, we see that Conjecture A holds for this case.
(Alternatively, we see that $|A(F)|=1$ and $e(F)=2$, then 
$\lambda (F)=0$ by Proposition \ref{lambda=0}.
Hence we can also apply Theorem \ref{thm_im1}.)
\end{example}

\begin{example}\label{ex2}
Put $k=\mathbb{Q} (\sqrt{-274})$ and $p=3$.
This pair satisfy (C1) and (C2).
We see that $|A(k)|=3$ and $\lambda (k)=4$.
Take $K$ and $F$ similarly ($K$ is uniquely determined).
We found that $|A(F_1)|=|D(F_1)|=3$.
Hence Conjecture A follows from Theorem \ref{thm_im2}.
Moreover, we see that $|A(F)|=1$ and $e(F)=2$, and then 
$\lambda (F)=0$ by Proposition \ref{lambda=0}.
From this, by using Theorem \ref{im_non_D}, we also see that 
$\mathcal{X} (k_\infty)$ is not a Demu{\v s}kin group.
\end{example}

At this time, Theorem \ref{thm_im2} seems more useful than Theorem \ref{thm_im1} (when $p=3$).
If there were an example of trivial $X (F_\infty)$ (with $\lambda (k) \geq 2$), 
it could be an example 
such that Theorem \ref{thm_im1} is applicable but Theorem \ref{thm_im2} is not applicable.
Unfortunately, there are no such examples (Theorem \ref{prop_conj1}).
Note also that \cite[Theorem 1 (2)]{Fuk94} is applicable  
to find an upper bound of $\lambda (F)$, even when $D(F_n)$ is trivial.
(That is, if $r_p (A(F_n)) =r_p (A(F_{n+1}))$ for some $n$, then $\lambda (F) \leq r_p (A(F_n))$.)
However, we need the information of $A(F_n)$ for a sufficiently large $n$
(in particular, when $|A(F)|=1$ and $|A(F_1)| > 1$, we must compute 
$A(F_n)$ for $n \geq 2$).

\subsection{Comparison of Theorem \ref{thm_im2} and Ozaki's criterion}\label{Ozaki}
We shall recall Ozaki's result \cite[Proposition 5]{Oza07}. 
In this paragraph, let $p$ be an odd prime number which does not split in $k$ 
(imaginary quadratic field).
Let $L^{(2)} (k_n)$ be the maximal unramified abelian $p$-extension over $L(k_n)$ 
such that 
\begin{itemize}
\item $L^{(2)} (k_n)$ is a Galois extension over $k_n$, and 
\item $\Gal(L^{(2)} (k_n)/L(k_n))$ is 
contained in the center of $\Gal (L^{(2)} (k_n) / k_n)$.
\end{itemize}
We put $X^{(2)} (k_n) = \Gal(L^{(2)} (k_n) / L(k_n))$.
Then $X^{(2)} (k_n)$ is the same as $X^{(2)}_n$ defined in \cite[p.62]{Oza07}.
We also denote by $D^{(2)} (k_n)$ the subgroup of $X^{(2)} (k_n)$ which is 
generated by the decomposition subgroups for the primes lying above $p$.
Then, Ozaki's result says that $\mathcal{X} (k_\infty)$ is not a free pro-$p$ group 
if both of the following conditions are satisfied: 
\begin{itemize}
\item $D^{(2)} (k_n)$ is not trivial for some $n$, and  
\item a certain condition on the $p$-adic $L$-function.
\end{itemize}
Actually, \cite[Proposition 5]{Oza07} gives a more precise result 
for the structure of $\mathcal{X} (k_\infty)$.

As remarked before, the above criterion seems close to our Theorem \ref{thm_im2}.
Note that Theorem \ref{thm_im2} does not need the condition on the $p$-adic $L$-function 
(but only works for the case when $p=3$).
We also mention that the non-triviality of $D(F_n)$ does not imply the 
non-triviality of $D^{(2)} (k_n)$ (for a fixed $n$).
In the following, we shall give an example.

\begin{example}\label{ex_9934}
Put $k= \mathbb{Q} (\sqrt{-9934})$ and $p=3$.
We note that $A(k) \cong (\mathbb{Z}/3 \mathbb{Z})^{\oplus 2}$.
Let $K$ be the splitting field of the polynomial 
\[ x^3 - x^2 - 39x - 109 \]
over $\mathbb{Q}$ 
(the author used PARI/GP to obtain this polynomial).
Then $K/k$ is an unramified cyclic cubic extension.
Let $F$ be a cubic extension of $\mathbb{Q}$ which is contained in $K$.
Then, $|A(F)|=9$ and $|D(F)|=3$ 
(it was checked by both PARI/GP and Magma).

On the other hand, we can see that $L^{(2)} (k) / L(k)$ is a cyclic cubic extension 
from the fact that $A (k) \wedge A(k) \cong \mathbb{Z}/3 \mathbb{Z}$
(see \cite{Oza07}).
Let $L^{sp} (K)/K$ be the maximal unramified abelian $3$-extension in which every 
prime lying above $3$ splits completely.
By using Magma,
it was checked that  
\[ \Gal (L^{sp} (K)/K) \cong \mathbb{Z}/9 \mathbb{Z} \oplus \mathbb{Z}/3 \mathbb{Z}. \]
$\Gal (L^{sp} (K)/L(k))$ is non-trivial.
We note that $L^{sp} (K)$ is a Galois extension over $k$.
Then the quotient of $\Gal (L^{sp} (K)/L(k))$ such that $\Gal (L(k)/k)$ acts trivially 
is not trivial.
Since $|\Gal (L^{(2)} (k)/L(k))|=3$, we can conclude that $L^{(2)} (k)$ is an intermediate field of 
$L^{sp} (K)/L(k)$.
Recall that every prime lying above $3$ splits completely in 
$L^{sp} (K) /L(k)$. 
Thus we see that $D^{(2)} (k)$ is trivial.

From the fact that $|D(F)|=3$, we see that 
$\mathcal{X} (k_\infty)$ is not a free pro-$3$ group by using Theorem \ref{thm_im2}.
However, since $D^{(2)} (k)$ is trivial, Ozaki's criterion is not applicable 
at least for $n=0$.
\end{example}

\begin{rem}\label{rem_GRH}
For the case when $p=3$ and 
$k = \mathbb{Q} (\sqrt{-211})$ or $k = \mathbb{Q} (\sqrt{-274})$ 
(recall Examples \ref{ex1} and \ref{ex2}), 
the author computed the ideal class group of $K_1$ by using Magma assuming GRH. 
(That is, the command \texttt{ClassGroup} was used with \texttt{Proof:="GRH"}.
Note that the degree of $K_1/\mathbb{Q}$ is $18$.)
We define $L^{sp} (K_1) /K_1$ similarly to the above.
For both cases, we found that $L^{sp} (K_1) = L(k_1)$ and $L (K_1) \neq L^{sp} (K_1)$. 
From these facts, we can deduce that $D^{(2)} (k_1)$ is not trivial (under GRH).
\end{rem}

See also the computational results given in \cite[pp.83--85]{Oza07}
(our computations given in 
Example \ref{ex_9934} and Remark \ref{rem_GRH} were done 
based on Ozaki's computation).

\section{Proof of Theorem \ref{relationship_GC} ($p \geq 2$)}\label{sec_real_GC}
In this section, $p$ denotes an arbitrary prime number. 
(In particular, we allow the case when $p=2$.)
Let the notation ($k', S, \mathbb{L}, \mathcal{G}$) be as in Theorem \ref{relationship_GC}.
First, we will show Theorem \ref{relationship_GC} when 
$\mathcal{G}$ satisfies the condition (D) 
(i.e., $\mathcal{G}$ is a Demu{\v s}kin group with $d (\mathcal{G}) \geq 3$).

\subsection{When $\mathcal{G}$ satisfies (D)}
Assume that $\mathcal{G} = \Gal (\mathbb{L}/k'_\infty)$ satisfies (D).
We need some preparations to state our proof.
Denote by $\mathbb{M}/ k'_\infty$ a finite extension contained in $\mathbb{L}$.
Let $\mathbb{L}_{\mathbb{M}}^{ab} /\mathbb{M}$ be the maximal abelian pro-$p$ extension 
contained in $\mathbb{L}$. 
We recall that $S (\mathbb{M})$ is the set of primes of $\mathbb{M}$ lying above $S$.
(Note that $S (\mathbb{M})$ is a finite set.)
For $v \in S (\mathbb{M})$, let $I (\mathbb{M}, v)$ be the 
inertia subgroup of $\Gal (\mathbb{L}_{\mathbb{M}}^{ab} /\mathbb{M})$ for $v$.
Since $v$ is not lying above $p$, we see that 
$I (\mathbb{M}, v)$ is finite or isomorphic to $\mathbb{Z}_p$ 
(see, e.g., \cite{NSW}).
We remark that if $I (\mathbb{M}, v) \cong \mathbb{Z}_p$, then 
$I (\mathbb{M}', v') \cong \mathbb{Z}_p$ for every finite extension 
$\mathbb{M}'/ \mathbb{M}$ contained in $\mathbb{L}$ and every prime $v' \in S (\mathbb{M}')$ 
lying above $v$.

We denote by $S_1 (k'_\infty)$ the maximal subset of $S (k'_\infty)$ 
which satisfies the following property: 
\begin{itemize}
\item[(I)] for every $v \in S_1 (k'_\infty)$ there exists a finite extension 
$\mathbb{M}/k'_\infty$ contained in $\mathbb{L}$ 
and a prime $v' \in S(\mathbb{M})$ lying above $v$ such that 
$I (\mathbb{M}, v') \cong \mathbb{Z}_p$.
\end{itemize}
We also put $S_2 (k'_\infty) = 
S (k'_\infty) \backslash S_1 (k'_\infty)$. 
That is, 
\begin{itemize}
\item[(II)] if $v \in S_2 (k'_\infty)$, then $I (\mathbb{M}, v')$ is finite  
for any finite extension $\mathbb{M}/k'_\infty$ contained in 
$\mathbb{L}$ and any prime $v' \in S (\mathbb{M})$ lying above 
$v \in S_2 (k'_\infty)$.
\end{itemize}
For a finite extension $\mathbb{M}/k'_\infty$ contained in $\mathbb{L}$, 
we denote by $S_1 (\mathbb{M})$ (resp. $S_2 (\mathbb{M})$) the set of primes of $\mathbb{M}$ 
lying above $S_1 (k'_\infty)$ (resp. $S_2 (k'_\infty)$).
Of course, $S (\mathbb{M}) = S_1 (\mathbb{M}) \cup S_2 (\mathbb{M})$ and 
$S_1 (\mathbb{M}) \cap S_2 (\mathbb{M}) = \emptyset$.

\begin{lem}\label{lem_K0}
Let the notation be as in Theorem \ref{relationship_GC}.
We also assume that $\mathcal{G}$ satisfies (D) of Theorem \ref{relationship_GC}.
Then, there is a 
finite extension $\mathbb{K}^{(0)}/k'_\infty$ contained in $\mathbb{L}$ 
satisfying the following properties:
\begin{itemize}
\item $I (\mathbb{K}^{(0)}, v') \cong \mathbb{Z}_p$ 
for every $v' \in S_1 (\mathbb{K}^{(0)})$, and 
\item $I (\mathbb{K}^{(0)}, v')$ is finite  
for every $v' \in S_2 (\mathbb{K}^{(0)})$.
\end{itemize}
\end{lem}

\begin{proof}
Take $v \in S_1 (k'_\infty)$ and 
$\mathbb{M}$ which satisfy $I (\mathbb{M}, v') \cong \mathbb{Z}_p$ for a prime $v'$ of 
$\mathbb{M}$ lying above $v$.
Then the Galois closure $\mathbb{M} (v)$ of $\mathbb{M}$ over $k'_\infty$ 
is a finite extension of $k'_\infty$ contained in $\mathbb{L}$, 
and satisfies the following property:
$I (\mathbb{M}(v), v'') \cong \mathbb{Z}_p$ for every 
$v'' \in S (\mathbb{M}(v))$ lying above $v$.

We put $\mathbb{K}^{(0)} = \bigcup_{v \in S_1 (k'_\infty)} \mathbb{M} (v)$, 
then $\mathbb{K}^{(0)}$ satisfies the desired properties.
\end{proof}

Let $\mathbb{K}^{(0)}$ be as in Lemma \ref{lem_K0}.
Put $U^{(0)} = \Gal (\mathbb{L}/\mathbb{K}^{(0)})$.
By Lemma \ref{pro-ell} (ii), $U^{(0)}$ is also a Demu{\v s}kin group, and satisfies 
\[ d (U^{(0)}) - 2 = (\mathcal{G} : U^{(0)}) (d(\mathcal{G}) -2 ). \]
Since $d(\mathcal{G}) \geq 3$, we also see that $d (U^{(0)}) \geq 3$.
We note that for every finite extension $\mathbb{M}'/\mathbb{K}^{(0)}$ contained 
in $\mathbb{L}$, $I (\mathbb{M}', v') \cong \mathbb{Z}_p$ 
(resp. $I (\mathbb{M}', v')$ is finite) for every 
$v' \in S_1 (\mathbb{M}')$ (resp. $v' \in S_2 (\mathbb{M}')$).

\begin{lem}\label{lem_K_infty}
Let the notation be as in Theorem \ref{relationship_GC}.
Assume that $\mathcal{G}$ satisfies (D) of Theorem \ref{relationship_GC}, 
and take $\mathbb{K}^{(0)}$ as in Lemma \ref{lem_K0}.
Then, there is an infinite Galois 
extension $\mathbb{K}^{(\infty)}/\mathbb{K}^{(0)}$ 
contained in $\mathbb{L}$ 
such that 
\begin{itemize}
\item $\Gal (\mathbb{K}^{(\infty)}/\mathbb{K}^{(0)}) \cong \mathbb{Z}_p$, and 
\item for every $v' \in S_1 (\mathbb{K}^{(0)})$, the inertia subgroup 
of $\Gal (\mathbb{K}^{(\infty)}/\mathbb{K}^{(0)})$ for $v'$ has finite index.
\end{itemize} 
\end{lem}

\begin{proof}
Recall that $\mathbb{L}_{\mathbb{K}^{(0)}}^{ab}/\mathbb{K}^{(0)}$ 
is the maximal abelian pro-$p$ extension contained in $\mathbb{L}$. 
We also note that 
$\Gal (\mathbb{L}_{\mathbb{K}^{(0)}}^{ab}/\mathbb{K}^{(0)})$ is finitely generated 
as a $\mathbb{Z}_p$-module.
Let $\widetilde{\mathbb{K}}$ be  
the intermediate field of 
$\mathbb{L}_{\mathbb{K}^{(0)}}^{ab}/\mathbb{K}^{(0)}$ corresponding to the 
maximal $\mathbb{Z}_p$-torsion subgroup of 
$\Gal (\mathbb{L}_{\mathbb{K}^{(0)}}^{ab}/\mathbb{K}^{(0)})$.
Then $\Gal (\widetilde{\mathbb{K}}/\mathbb{K}^{(0)})$ 
is isomorphic to $\mathbb{Z}_{p}^{\oplus m}$ 
(where $m$ is equal to $d (U^{(0)})$ or $d (U^{(0)})-1$, and then $m \geq 2$).
Since $\mathbb{L}_{\mathbb{K}^{(0)}}^{ab}/\widetilde{\mathbb{K}}$ is finite, 
the inertia subgroup of $\Gal (\widetilde{\mathbb{K}}/\mathbb{K}^{(0)})$ for 
every $v' \in S_1 (\mathbb{K}^{(0)})$ is isomorphic to $\mathbb{Z}_p$.

Then the assertion follows by using Kataoka's result \cite[Lemma 3.3]{Kat} 
(by setting $M$ as $\Gal (\widetilde{\mathbb{K}}/\mathbb{K}^{(0)})$ and $L_j$ as the 
inertia subgroups). 
Actually, we only need a much weaker result than Kataoka's, 
because it is sufficient to find one extension $\mathbb{K}^{(\infty)}/\mathbb{K}^{(0)}$.
\end{proof}

\begin{proof}[Proof of Theorem \ref{relationship_GC} when $\mathcal{G}$ satisfies (D)]
Take $\mathbb{K}^{(0)}$ and $\mathbb{K}^{(\infty)}$ be as in above lemmas.
Then there is a sequence of extensions 
\[ \mathbb{K}^{(0)} \subset \mathbb{K}^{(1)} \subset \mathbb{K}^{(2)}
\subset \cdots \subset \mathbb{K}^{(n)} \subset \cdots \subset \mathbb{K}^{(\infty)} \]
such that 
\begin{itemize}
\item $\Gal (\mathbb{K}^{(n)}/\mathbb{K}^{(0)}) \cong \mathbb{Z}/ p^n \mathbb{Z}$ 
for all $n$, and 
\item if $n$ is sufficiently large, then 
every $v' \in S_1 (\mathbb{K}^{(0)})$ actually ramifies in $\mathbb{K}^{(n)}/\mathbb{K}^{(0)}$.
\end{itemize}
Let $s_n$ be the cardinality of $S_1 (\mathbb{K}^{(n)})$.
The above facts imply that $s_n$ is bounded as $n \to \infty$.
Fix a positive integer $C$ satisfying $s_n < C$ for all $n$.

We put $U^{(n)} = \Gal (\mathbb{L}/\mathbb{K}^{(n)})$, 
and $\mathbb{L}^{(n),ab} = \mathbb{L}_{\mathbb{K}^{(n)}}^{ab}$ 
(the maximal abelian pro-$p$ extension of $\mathbb{K}^{(n)}$ contained in $\mathbb{L}$).
By using Lemma \ref{pro-ell} (ii), we see that $U^{(n)}$ is also a Demu{\v s}kin group and 
\[ d (U^{(n)}) - 2 = p^n (d(U^{(0)}) -2 ). \]
Recall that $d(U^{(0)}) \geq 3$, 
and then an inequality $d (U^{(n)}) \geq p^n + 2$ holds.
Since $\Gal (\mathbb{L}^{(n),ab} /\mathbb{K}^{(n)}) \cong (U^{(n)})^{ab}$, 
the $\mathbb{Z}_p$-rank of $\Gal (\mathbb{L}^{(n),ab} /\mathbb{K}^{(n)})$ is 
greater than $p^n$.

Let $\mathcal{I}^{(n)}$ be the $\mathbb{Z}_p$-submodule of 
$\Gal (\mathbb{L}^{(n),ab} /\mathbb{K}^{(n)})$ 
generated by $I (\mathbb{K}^{(n)}, v)$ for all $v \in S (\mathbb{K}^{(n)})$.
Since $\mathbb{L}^{(n),ab} /\mathbb{K}^{(n)}$ is unramified outside $S (\mathbb{K}^{(n)})$, 
the fixed field $\mathbb{L}^{(n),ab}_{unr}$ of $\mathcal{I}^{(n)}$ is an unramified 
abelian pro-$p$ extension over $\mathbb{K}^{(n)}$.
By using the fact which is noted before, we see that 
$I (\mathbb{K}^{(n)}, v) \cong \mathbb{Z}_p$ 
(resp. $I (\mathbb{K}^{(n)}, v)$ is finite) for every 
$v' \in S_1 (\mathbb{K}^{(n)})$ (resp. $v' \in S_2 (\mathbb{K}^{(n)})$).
Hence the $\mathbb{Z}_p$-rank of $\mathcal{I}^{(n)}$ is smaller than $C$.

Take $n$ such that $p^n > C$.
By combining the above facts, we see that 
the $\mathbb{Z}_p$-rank of $\Gal (\mathbb{L}^{(n),ab}_{unr} /\mathbb{K}^{(n)})$ is positive.
This implies that $X (\mathbb{K}^{(n)})$ is infinite.
We recall that $\mathbb{K}^{(n)} / k'_\infty$ is a finite extension.
As noted in \cite{Iwa81} (without proof), 
there exists a finite extension $k''$ of $\mathbb{Q}$ such that 
$\mathbb{K}^{(n)}$ is the cyclotomic $\mathbb{Z}_p$-extension of $k''$.
Since any archimedean prime does not ramify in 
$\mathbb{K}^{(n)}/k'_\infty$, we see that this $k''$ is totally real.
Then the theorem (when $\mathcal{G}$ satisfies (D)) follows.
\end{proof}

\subsection{When $\mathcal{G}$ satisfies (F)}
Next, we will show the remaining case.

\begin{proof}[Proof of Theorem \ref{relationship_GC} when $\mathcal{G}$ satisfies (F)]
In fact, we can also see the assertion for this case quite similarly 
(by using Theorem \ref{pro-ell} (i) instead of Theorem \ref{pro-ell} (ii)). 

As an alternative proof, we can see that the existence of the objects satisfying (F)
implies the existence of the objects satisfying (D).
That is, if there is a triple $k'$, $S$, $\mathbb{L}$ satisfying (F), then 
we can take intermediate fields $\mathbb{M}$ and $\mathbb{L}'$ of $\mathbb{L}/k'_\infty$ 
such that 
$\mathbb{M}/k'_\infty$ is a finite extension and $\Gal (\mathbb{L}'/\mathbb{M})$ is a 
Demu{\v s}kin group with $d(\Gal (\mathbb{L}'/\mathbb{M})) \geq 3$.
(Note that $\mathbb{M}$ is the cyclotomic $\mathbb{Z}_p$-extension of a 
certain totally real field.)
\end{proof}

\section{Proof of Theorem \ref{thm_real_quad} ($p \geq 3$)}\label{sec_real_quadratic}

Let $k'$ and $S$ be as in Theorem \ref{thm_real_quad}.
We assume that $p$ is an odd prime number such that $p$ is inert in $k'$ and 
does not divide $h(k')$.
(Hence $X (k'_\infty)$ is trivial by Iwasawa's theorem \cite{Iwa56}.)
Moreover, we assume that $X_{\{q_1\}} (k')$ and $X_{\{q_2 \}} (k')$ are trivial.

\subsection{Several preliminary facts}\label{real_quad_pre}

\begin{lem}\label{lem_real1}
Let the assumptions be as in Theorem \ref{thm_real_quad}.
Then 
\[ X_{\{q_1, q_2 \}} (k') \cong \mathbb{Z}/p \mathbb{Z},\; 
X_{S} (k') \cong (\mathbb{Z}/p \mathbb{Z})^{\oplus r-1}, \;
X_{\{q_1, q_2 \}} (k'_\infty) \cong \mathbb{Z}_p, \; \text{and} \; 
X_S (k'_\infty) \cong \mathbb{Z}_p^{\oplus r-1} \]
as $\mathbb{Z}_p$-modules.
\end{lem}

\begin{proof}
Put $R_{k'} (q_i) = (O_{k'}/ q_i O_{k'})^\times \otimes_{\mathbb{Z}} \mathbb{Z}_p$ for 
$i=1, \ldots, r$ (where $O_{k'}$ is the ring of integers of $k'$).
We note that $R_{k'} (q_i)$ is a cyclic group of order $p$.
Then, by class field theory, we obtain the exact sequence
\[ E(k') \otimes_{\mathbb{Z}} \mathbb{Z}_p \overset{f}{\to} 
R_{k'} (q_1) \oplus R_{k'} (q_2) \to X_{\{q_1, q_2 \}} (k') \to 0 \]
(recall that $X (k')$ is trivial).
By the assumption that $X_{\{ q_1 \}} (k')$ is trivial, we see that 
the natural mapping $E(k') \otimes_{\mathbb{Z}} \mathbb{Z}_p \to R_{k'} (q_1)$ is surjective.
From this, we see that the cokernel of $f$ is a cyclic group of order $p$ 
(note that the image of $f$ is cyclic).
The assertion that $X_{\{q_1, q_2 \}} (k') \cong \mathbb{Z}/p \mathbb{Z}$ follows.
The assertion that $X_S (k') \cong (\mathbb{Z}/p \mathbb{Z})^{\oplus r-1}$ can be 
shown similarly.
(See also Remark 4.3 of \cite{Itoh18}.)

The remaining two assertions can be shown by using \cite[Corollary 4.2]{Itoh18}.
Note that the $\mathbb{Z}_p$-rank can be computed by using the formula given in 
\cite[Example 6.6]{Itoh14} 
(see also the proof of \cite[Theorem 1.2]{Itoh18}).
\end{proof}

We note that $L_{\{ q_1, q_2 \}} (k')$ is a Galois extension over $\mathbb{Q}$ of degree $2p$.
We can see that $\Gal (L_{\{ q_1, q_2 \}} (k')/\mathbb{Q})$ is isomorphic to the dihedral group 
of order $2p$.
(If it is a cyclic group, then there exists a cyclic extension 
of degree $p$ over $\mathbb{Q}$ unramified outside $\{ q_1, q_2 \}$.
However, since $q_1 \equiv q_2 \equiv -1 \pmod{p}$,
such an extension does not exist.)

\begin{definition}\label{def_KH}
Let the notation be as in Theorem \ref{thm_real_quad}.
For an integer $n \geq 1$, we put $\zeta_{p^n} = e^{\frac{2 \pi \sqrt{-1}}{p^n}}$.
We also put 
\[ K' = L_{\{ q_1, q_2 \}} (k'),  \quad H = K' (\zeta_p), \quad \text{and} \quad 
H^+ = K' (\zeta_p + \zeta_p^{-1}). \]
Since $K'$ is totally real, $H$ is a CM-field and 
$H^+$ is the maximal real subfield of $H$.
\end{definition}

From the facts that $p$ is inert in $k'$ and 
$\Gal (K'/\mathbb{Q})$ is isomorphic to the dihedral group of order $2p$, 
we see that the unique prime of $k'$ lying above $p$ splits completely in $K'$.
Similarly, since $q_j \, (j =3, \ldots, r)$ is also inert in $k'$, 
then the unique prime of $k'$ lying above $q_j$ splits completely in $K'$.
On the other hand, since $X_{\{ q_1 \}} (k')$ (resp. $X_{\{ q_2 \}} (k')$) is trivial, 
the unique prime of $k'$ lying above $q_1$ (resp. $q_2$) must ramify in $K'$.

We also remark that every prime of $K'$ lying above $p$ is totally ramified in $K'_\infty/K'$, 
and every prime of $K'$ lying above $q_i \, (i =1, \ldots, r)$ does not split 
in $K'_{\infty} /K'$.
(Recall the assumptions that $q_i \equiv -1 \pmod{p}$ and 
$q_i^2 \not\equiv 1 \pmod{p^2}$. 
From these, 
$q_i$ does not split in $\mathbb{B}_\infty/\mathbb{Q}$.)

\begin{lem}\label{lem_real2}
Let the assumptions be as in Theorem \ref{thm_real_quad}, and 
$K'$ as in Definition \ref{def_KH}.
Then $X(K'_\infty)$ is trivial.
\end{lem}

\begin{proof}
We note that $K'_\infty$ is the unique intermediate field of 
$L_{\{q_1, q_2\}} (k'_\infty) /k'_\infty$ of degree $p$ over $k'_\infty$ 
(recall that $X_{\{q_1, q_2\}} (k'_\infty) \cong \mathbb{Z}_p$ 
by Lemma \ref{lem_real1}).
Moreover, we also see that 
$\mathcal{L}_{\{q_1, q_2\}} (k'_\infty) =L_{\{q_1, q_2\}} (k'_\infty)$.
Note that the prime of $k'_\infty$ lying above $q_1$ is totally ramified in 
$L_{\{q_1, q_2\}} (k'_\infty)$
(this follows from the fact that the prime of $k'$ lying above $q_1$ ramifies in $K'$ but 
does not ramify in $k'_\infty$).
Hence, if $X (K'_\infty)$ is not trivial, then there is a non-trivial 
unramified abelian $p$-extension of $L'/K'_\infty$ such that 
$L' \cap L_{\{q_1, q_2\}} (k') = K'_\infty$.
However, this contradicts the fact that 
$\mathcal{L}_{\{q_1, q_2\}} (k'_\infty) =L_{\{q_1, q_2\}} (k'_\infty)$.
Hence $X (K'_\infty)$ is trivial, and then the assertion follows.
\end{proof}

To show the assertion of Theorem \ref{thm_real_quad}, it is sufficient to show that 
the $\mathbb{Z}_p$-rank of $X_S (K'_\infty)$ is not so large.

\begin{prop}\label{no_free_prop}
Let the assumptions be as in Theorem \ref{thm_real_quad}, and 
$K'$ as in Definition \ref{def_KH}.
If 
\[ \rkzp X_S (K'_\infty) \leq p (r-2) \]
then $\mathcal{X}_S (k'_\infty)$ is not a free pro-$p$ group.
\end{prop}

\begin{proof} 
Put $\mathcal{G} = \mathcal{X}_S (k'_\infty)$ and $U = \mathcal{X}_S (K'_\infty)$.
From the fact that $K'/k'$ is unramified outside $S(k')$, 
we see that $\mathcal{L}_S (K'_\infty) = \mathcal{L}_S (k'_\infty)$, and hence 
$U$ is an open subgroup of $\mathcal{G}$ of index $p$.

Assume that $\mathcal{G}$ is a free pro-$p$ group.
Since $X_S (k'_\infty) \cong \mathbb{Z}_p^{\oplus r-1}$, we see that $d (\mathcal{G}) = r-1$.
Hence, by using Lemma \ref{pro-ell} (i), we see that 
\[ \rkzp X (K'_\infty) = d(U) = p (r-2)+1. \] 
The assertion follows from this.
\end{proof}

\subsection{Proof of Theorem \ref{thm_real_quad}}
The following proposition is a key result to prove Theorem \ref{thm_real_quad}.

\begin{prop}\label{real_key}
Let the assumptions be as in Theorem \ref{thm_real_quad}, and 
$H^+$ as in Definition \ref{def_KH}.
Then, $\Gal (L_{\{ q_3 \}} (H^+_\infty) / L (H^+_\infty))$ is finite.
\end{prop}

\begin{proof}
Our proof of this proposition is almost same as that of 
a special case of \cite[Theorem 1.1]{Itoh14}
(see also the proof of \cite[Theorem 3.1]{IMO}).
Hence, for the most part of this proof, 
we will state the outline only.

However, since $H^+$ is not an abelian extension over $\mathbb{Q}$, 
we must be careful in some places.
In the proof of \cite[Theorem 1.1]{Itoh14}, the following are major places 
in which the assumption that the base field is a real abelian field is used.
\begin{itemize}
\item Calculation of the $\mathbb{Z}_p$-rank of the subgroup generated by 
the inertia subgroups of the Galois group of a certain Kummer extension 
(see \cite[Proposition 3.3]{Itoh14}).
\item Estimation of (a part of) the characteristic ideal of the ``$p$-ramified Iwasawa module''  
(see \cite[pp.521--522]{Itoh14}).
\end{itemize}
However, for our case, we can also handle these places.

Let $H$ be as in Definition \ref{def_KH}.
Take a topological generator $\gamma$ of $\Gal (H_\infty / H)$ 
which satisfies $\zeta_{p^{n}}^\gamma = \zeta_{p^n}^{1+p}$ for all $n$.
(We will use the same symbol $\gamma$ for the restriction of $\gamma$ to $H^+_\infty$.)
We also fix an isomorphism from $\mathbb{Z}_p [[ \Gal (H_\infty / H) ]]$ 
to $\Lambda = \mathbb{Z}_p [[T]]$ satisfying $\gamma \mapsto 1+T$, 
and we regard $\mathbb{Z}_p [[ \Gal (H_\infty / H) ]]$-modules 
(or $\mathbb{Z}_p [[ \Gal (H^+_\infty / H^+) ]]$-modules) as 
$\Lambda$-modules via this isomorphism.
For a finitely generated torsion $\Lambda$-module $Y$, we denote 
by $\charideal Y$ its characteristic ideal.

We put $m = p(p-1)/2$ ($=$ the degree of $H^+/k'$).
Hence the degree of $H/k'$ is $2m$.
From the assumptions on $q_3$, we see that  
the unique prime of $k'$ lying above $q_3$ splits completely in $H$.
(Moreover, every prime lying above $q_3$ does not split in $H_\infty /H$.)
We also note that there are $p$ primes of $H$ lying above $p$, 
and denote by $\mathfrak{P}_1, \ldots, \mathfrak{P}_p$ these primes.
For $i = 1, \ldots, p$, 
let $\mathcal{U}_{\mathfrak{P}_i}^1 (H)$ be the group of principal units of $H_{\mathfrak{P}_i}$ 
(see also the proof Theorem \ref{thm_im2} given in Section \ref{proof_thm_im2}).
We put $\mathcal{U}^1 = \prod_{i=1}^p \mathcal{U}_{\mathfrak{P}_i}^1 (H)$.
Then the $\mathbb{Z}_p$-rank of $\mathcal{U}^1$ is $2 p (p-1) = 4m$.

Fix a prime $\mathfrak{Q}$ of $H$ which is lying above $q_3$.
We can take a positive integer $h$ such that 
$\mathfrak{Q}^h$ is a principal ideal generated by $\alpha$ 
(where $\alpha$ is an algebraic integer of $H$), and 
$\alpha \equiv 1 \pmod{\mathfrak{P}_i}$ for all $\mathfrak{P}_i$ ($i=1, \ldots, p$).
Let $\iota$ be the complex conjugation. 
Take elements $\sigma_1, \ldots, \sigma_{m}$ 
of $\Gal (H/k')$ which represents $\Gal (H/k')$ modulo $\langle \iota \rangle$.
We define a subgroup $A$ of the multiplicative group $H^\times$ generated by  
\[ \alpha^{\sigma_1}, \alpha^{\sigma_1 \iota}, \ldots, 
\alpha^{\sigma_m}, \alpha^{\sigma_m \iota}. \]
We see that $A$ is free of rank $2m$ as a $\mathbb{Z}$-module,
because the unique prime of $k'$ lying above $q_3$ splits completely in $H$.

We put $\beta_i = \alpha^{\sigma_i}/\alpha^{\sigma_i \iota}$ and 
\[ N = \bigcup_{n \geq 1} H_\infty (\sqrt[p^n]{\beta_1}, 
\ldots, \sqrt[p^n]{\beta_m}). \]
Then $N$ is a Galois extension over $H^+$, and an abelian extension over $H^+_\infty$.
Let $\mathcal{I}$ be the subgroup of $\Gal (N/H_\infty)$ which is generated by the inertia 
subgroups for the primes lying above $p$.
We would like to show that $\rkzp \mathcal{I} = m$
by using \cite[Lemma 2.5]{Kha-Win} (see also \cite{IMO}).
To see this, it is sufficient to show that 
the $\mathbb{Z}_p$-rank of $\mathcal{A}$ is $2m$, 
where $\mathcal{A}$ is the closure of the image of the diagonal embedding 
$A \to \mathcal{U}^1$.
In \cite{Itoh14}, the assumption that the base field is abelian over $\mathbb{Q}$ is 
used to prove a similar result (see also \cite{IMO}).
However, in our case, we only consider the $\Gal (H/k')$-conjugates of $\alpha$, and
$H/k'$ is an abelian extension.
Thanks to these facts, we can imitate the argument (using Baker-Brumer's theorem \cite{Brumer}) 
given in \cite[p.1498]{IMO}, \cite[pp.519--520]{Itoh14}.
(Note that we consider the matrix 
$( \log_p (\alpha^{\sigma \tau^{-1}}))_{\sigma, \tau \in \Gal (H/k')}$ in our case.
See also \cite[\S 5.6]{Was}.)
Thus we see that $\rkzp \mathcal{A} = 2m$, and hence 
\begin{equation}\label{eq_rankI}
\rkzp \mathcal{I} = m.
\end{equation}
Let $N^+$ be the unique intermediate field of $N/H^+_\infty$ such that 
$\Gal (N^+/H^+_\infty) \cong \mathbb{Z}_p^{\oplus m}$.
From the above facts, $N^+/H^+_\infty$ is unramified outside the primes lying above $p$ or $q_3$, 
and $\gamma$ acts on $x \in \Gal (N^+/H^+_\infty)$ as $x^{\gamma} = x^{1+p}$.

Let $M_{\{ p, q_3 \}} (H^+_\infty)/ H^+_\infty$ (resp. $M_{\{ p \}} (H^+_\infty)/ H^+_\infty$)
be the maximal abelian pro-$p$ extension 
which is unramified outside the primes lying above $p$ or $q_3$ (resp. $p$). 
Put 
\[ \mathfrak{X}_{\{ p, q_3 \}} (H^+_\infty) = \Gal (M_{\{ p, q_3 \}} (H^+_\infty)/ H^+_\infty) 
\quad \text{and} \quad \mathfrak{X}_{\{ p\}} (H^+_\infty) = 
\Gal (M_{\{ p \}} (H^+_\infty)/ H^+_\infty). \]
Since $H^+$ is totally real, both 
$\mathfrak{X}_{\{ p, q_3 \}} (H^+_\infty)$ and $\mathfrak{X}_{\{ p \}} (H^+_\infty)$ are 
finitely generated torsion $\Lambda$-modules 
(see, e.g., \cite[(11.3.2)Theorem (iv)]{NSW}).

We shall consider $\charideal \mathfrak{X}_{\{ p, q_3 \}} (H^+_\infty)$ 
(in \cite{Itoh14}, the assumption that the base field is a real abelian field is 
used in this process).
Before this, we consider $X (H_\infty)^- = (\iota - 1) X (H_\infty)$ 
(the usual minus part of the unramified Iwasawa module for a CM field).
Note that $k' (\zeta_p)$ is an abelian extension over $\mathbb{Q}$, and 
$H$ is a cyclic extension of degree $p$ over $k' (\zeta_p)$.
Then, by using Ferrero-Washington's theorem \cite{F-W} and 
Iwasawa's theorem \cite[Theorem 3]{Iwa73mu} (see also Section \ref{sub_Settings}), 
we see that $\mu (H)=0$.
Take $g(T) \in \Lambda$ which generates $\charideal X (H_\infty)^-$.
We note the important fact that there is no 
prime lying above $p$ which splits in $H_\infty/H^+_\infty$ 
(because $p$ does not ramify in $K'/\mathbb{Q}$).
This implies that $g(T)$ is relatively prime to $T$. 
(We remark that $H$ is not abelian over $\mathbb{Q}$, however, 
this assertion holds.
See, e.g., \cite{F-G}.)
Put $h(T) = g ( \frac{1+p}{1+T}-1) \in \Lambda$.
Then, by Kummer duality, we see that 
$h(T)$ generates $\charideal \mathfrak{X}_{\{ p \}} (H^+_\infty)$.
(Note that we can see that 
\[ \mathfrak{X}_{\{ p \}} (H^+_\infty) \cong 
\mathop{\mathrm{Hom}_{\mathbb{Z}_p}} (A(H_\infty)^-, W_{p^\infty}), \]
where $A(H_\infty)^-$ is the minus part of $\varinjlim A(H_n)$ and 
$W_{p^\infty}$ is the group of all $p$-power roots of unity.
Moreover, from a well known argument using the adjoint, 
we can obtain this fact.
See, e.g., \cite[Chapters 13, 15]{Was}.)
Hence we see that $\mathfrak{X}_{\{ p \}} (H^+_\infty)$ is finitely generated 
as a $\mathbb{Z}_p$-module, and $h(T)$ is relatively prime to $T-p$.
By using \cite[(11.3.5)Theorem]{NSW} (see also its proof), 
we can see that 
\begin{equation}\label{eq_charideal}
(T-p)^{m}\, h(T) \;\; \text{generates} \;\; \charideal \mathfrak{X}_{\{ p, q_3 \}} (H^+_\infty).
\end{equation} 
(Note that there are $m$ primes in $H^+$ lying above $q_3$, 
and they does not split in $H^+_\infty$. 
Actually, we can know the structure of $\mathfrak{X}_{\{ p, q_3 \}} (H^+_\infty)$ 
more precisely.)

The remaining part of the proof is similar to that given in \cite[pp.522--525]{Itoh14} 
(for ``Case NS'').
Let $M'$ be the intermediate field of $M_{\{ p, q_3 \}} (H^+_\infty)/L(H^+_\infty)$ 
corresponding to  
\[ (T-p) \, \Gal(M_{\{ p, q_3 \}} (H^+_\infty)/L(H^+_\infty)). \]
We see that both $L_{\{ q_3 \}} (H^+_\infty)$ and
$L(H^+_\infty) N^+$ are intermediate fields of $M'/L(H^+_\infty)$.
By using (\ref{eq_charideal}), we can show that $\rkzp \Gal (M'/L(H^+_\infty)) =m$.
Let $\mathcal{I}'$ be the subgroup of $\Gal (M'/H^+_\infty)$ 
generated by the inertia subgroups for the primes lying above $p$.
Then we can see $\rkzp \mathcal{I}' = m$ by using (\ref{eq_rankI}). 
Hence $L_{\{ q_3 \}} (H^+_\infty) /L (H^+_\infty)$ is a finite extension.
\end{proof}

\begin{cor}\label{key_cor}
Let the assumptions be as in Theorem \ref{thm_real_quad}, and 
$K'$ as in Definition \ref{def_KH}.
Then, $X_{\{ q_3 \}} (K'_\infty)$ is finite.
\end{cor}

\begin{proof}
Let $H^+$ be as in Definition \ref{def_KH}.
We see that 
$\Gal (L_{\{ q_3 \}} (H^+_\infty) /L (H^+_\infty))$ is finite 
by Proposition \ref{real_key}.
Then we can show that 
$\Gal (L_{\{ q_3 \}} (K'_\infty) /L (K'_\infty))$ is also finite 
(see, e.g., \cite[Lemma 2.1]{Itoh14}).
By Lemma \ref{lem_real2}, we see that 
$\Gal (L (K'_\infty) / K'_\infty)$ is trivial.
Thus the assertion follows.
\end{proof}

\begin{rem}
In the last line of the proof of \cite[Lemma 2.1]{Itoh14}, 
``$\Gal (M_q (k'_n)/L(k'_n))$'' should be read ``$\Gal (M_q (k_n)/L(k_n))$''.
\end{rem}

Now we shall finish the proof of Theorem \ref{thm_real_quad}.

\begin{proof}[Proof of Theorem \ref{thm_real_quad}] 
Let $K'$ be as in Definition \ref{def_KH}.
We shall give an upper bound of $\rkzp X_{S} (K'_\infty)$ by using 
(a rough version of) the known method (see, e.g., \cite{M-O13}, \cite{IMO}).

For $\mathfrak{q} \in S(K')$, we put 
$R_n (\mathfrak{q}) = 
(O_{K'_n} / \mathfrak{q}_{n})^\times \otimes_{\mathbb{Z}} \mathbb{Z}_p$, where 
$\mathfrak{q}_{n}$ is the extension of $\mathfrak{q}$ in $K'_n$.
We also put $R_\infty (\mathfrak{q}) = \varprojlim R_n (\mathfrak{q})$.
Since $|(O_{K'} / \mathfrak{q})^\times|$ is divisible by $p$ and 
$\mathfrak{q}$ is not decomposed in $K'_\infty$, 
we see that $R_\infty (\mathfrak{q}) \cong \mathbb{Z}_p$ (as a $\mathbb{Z}_p$-module)
for every $\mathfrak{q} \in S(K')$.

Recall that there is only one prime of $K'$ lying above $q_i$ (for $i=1,2$), 
and there are $p$ primes of $K'$ lying above $q_j$ (for $j=3, \ldots, r$).
We put 
\[ R_\infty (S) = \bigoplus_{\mathfrak{q} \in S(K')} R_\infty (\mathfrak{q})
\quad \text{and} \quad 
R_\infty (q_3) = \bigoplus_{\mathfrak{q} \in S(K'), \;  
\mathfrak{q} | q_3} R_\infty (\mathfrak{q}). \]
By using class field theory (and Lemma \ref{lem_real2}), 
we obtain the following exact sequences 
\[ \varprojlim (E (K'_n) \otimes_{\mathbb{Z}} \mathbb{Z}_p) \overset{f_1}{\to} 
R_\infty (S) \to X_{S} (K'_\infty) \to 0, \]
\[ \varprojlim (E (K'_n) \otimes_{\mathbb{Z}} \mathbb{Z}_p) \overset{f_2}{\to} 
R_\infty (q_3) \to X_{\{ q_3 \}} (K'_\infty) \to 0. \]
By Corollary \ref{key_cor}, $X_{\{ q_3 \}} (K'_\infty)$ is finite.
This implies that the $\mathbb{Z}_p$-rank of the image of $f_2$ is $p$ 
(recall that $\rkzp R_\infty (q_3) = p$). 
Then, we see that the $\mathbb{Z}_p$-rank of the image of $f_1$ is 
greater than or equal to $p$.
Since 
\[ \rkzp R_\infty (S) = p (r-2) + 2, \]
we have obtained the inequality 
\[ \rkzp X_{S} (K'_\infty) \leq p (r-2) + 2 - p. \]
Then Theorem \ref{thm_real_quad} follows from Proposition \ref{no_free_prop}.
\end{proof}

\begin{acknowledgements}
The author would express his thanks to Manabu Ozaki for giving motivation and advice.
The author also express thanks to Yutaka Konomi for giving information around his paper \cite{Kono}.
This work was partly supported by JSPS KAKENHI Grant Number JP15K04791.
\end{acknowledgements}

\begin{noteIT}
As mentioned in \cite{F-I}, there is a problem in the programs used 
in the computations of \cite[Section 5]{I-T}.
We explain the details.
The Magma command \texttt{RayClassGroup} returns a ``guaranteed'' result only when 
the ideal class group was computed in advance.
See \cite[Chapter 40]{MagmaHB}.
However, the process of computing the ideal class group in advance   
was omitted in the programs used in the research of \cite{I-T} 
(as was confirmed by the author himself).
Hence, it has to be said 
that the computational results given \cite[Section 5]{I-T} were 
``not guaranteed by Magma'' at the time when that paper was published.

Thus, the author re-computed the ray class groups by using a method to obtain 
results guaranteed by Magma.
(The author used the same defining polynomials of number fields 
which were used in the original programs.
Hence, the defining polynomials given in \cite{K-O}, \cite{Brink} 
were also used in this re-computation.)
Consequently, it was checked that the values given in the tables of \cite{I-T} 
coincide with the values obtained by this re-computation.
(Hence we do not need a correction of the values.)

Thanks to Shun'ichi Yokoyama and the 
Magma staff (especially Nicole Sutherland) for answering questions and giving advice.
The author wrote this note with the consent of the co-author of \cite{I-T}.
\end{noteIT}

\bigskip

\begin{flushleft}
Tsuyoshi Itoh \\
Division of Mathematics, 
Education Center,
Faculty of Social Systems Science, \\
Chiba Institute of Technology, \\
2--1--1 Shibazono, Narashino, Chiba, 275--0023, Japan \\
e-mail : \texttt{tsuyoshi.itoh@it-chiba.ac.jp}

\end{flushleft}

\end{document}